\documentclass[a4paper,11pt]{amsart}
\usepackage{amssymb,amsmath,amscd,url,hyperref,enumerate,slashed,units,graphicx,epstopdf}

\newtheorem{thm}{Theorem}[section]

\newtheorem{lem}[thm]{Lemma}

\newtheorem{rem}[thm]{Remark}

\newtheorem{defn}[thm]{Definition}

\newcommand{\Z}{\mathbb{Z}}
\newcommand{\R}{\mathbb{R}}
\newcommand{\C}{\mathbb{C}}
\newcommand{\hatD}{\widehat D}
\newcommand{\ds}{W}
\def\fsl{\mathfrak{sl}}
\def\SL{\mathrm{SL}}

\def\trace{\mathrm{trace}}
\def\fg{\mathfrak{g}}

\def\fA{\mathfrak{A}}
\def\Ahat{\widehat{\mathfrak{A}}}
\def\fB{\mathfrak{B}}
\def\fC{\mathfrak{C}}
\def\fD{\mathfrak{D}}
\def\fK{\mathfrak{K}}

\def\cE{\mathcal{E}}
\def\cZ{\mathcal{Z}}
\def\cY{\mathcal{Y}}
\def\cU{\mathcal{U}}
\def\Hdirac{\mathbb{H}}
\begin{document}
\title[Local spectrum]{The local spectrum of the  Dirac operator  for the universal cover of $\SL_2(\R)$}   

\author[Brodzki et al.]{Jacek Brodzki, Graham A. Niblo, Roger Plymen and Nick Wright}
\address{Mathematical Sciences, University of Southampton, SO17 1BJ,  England}
\email{j.brodzki@soton.ac.uk, g.a.niblo@soton.ac.uk, r.j.plymen@soton.ac.uk, wright@soton.ac.uk}

\date{\today}
\subjclass[2010]{}
\maketitle

\begin{abstract}   
Using representation theory, we compute the spectrum of the Dirac operator on  the universal covering group of $\SL_2(\R)$, exhibiting it as the generator 
of $KK^1(\C, \fA)$, where $\fA$ is the reduced $C^*$-algebra of the group. This yields a new and direct computation of the $K$-theory of $\fA$. 
A fundamental role is played by the limit-of-discrete-series representation, which is the frontier between the discrete and the principal series of the group. 
We provide a detailed analysis of the localised spectra of the Dirac operator
and compute the Dirac cohomology. 
 \end{abstract}


\section{Introduction}\label{introduction}    Let $G$ denote the universal cover of $\SL_2(\R)$ and     let $\fA = C^*_r(G)$ be the reduced $C^*$-algebra of $G$.     
The group $G$ is a non-linear group with infinite centre which places it outside the range of much classical representation theory of Harish-Chandra et al and in particular the $K$-theory of $\fA$ is known only through the deep results on the Connes-Kasparov conjecture due to Chabert, Echterhoff and Nest \cite{CEN}.  In this article, we give a direct computation of the $K$-theory of $\fA$ using the Plancherel formula for $G$, established  by Puk\'anszky \cite{P}, studying the algebra via the Fourier transform.  Moreover this also enables us to compute the spectrum of the Dirac operator on $G$, thereby establishing that this is the generator 
of $KK^1(\C, \fA)$. 

%

We use the Fourier transform to identify $\fA$ as an algebra of operator-valued functions on a parameter space built from the discrete series and the principal series of $G$. 
Here there is an analogy with the representation theory of $\SL_2(\R)$. This has a representation, the limit-of-discrete-series, which is a reducible representation in the 
principal series of $\SL_2(\R)$ sharing properties with representations from the discrete series. Now, for $G$, the analogue of the discrete series is a field of pairs of irreducible representations
parametrised by the interval $(-\infty, 1/4)$, whose limit at $1/4$ is again a reducible representation in the principal series. This limit representation we call the \emph{limit-of-discrete-series} for $G$. The principal series is parametrised by a cylinder $S^1\times [1/4, \infty)$ and it is the attaching of the two spaces at the limit-of-discrete-series that yields the 
generator in $K$-theory. We remark that the Casimir operator for $G$, which acts as a scalar on each irreducible representation, gives the real parameter in $(-\infty,1/4)$ for the 
discrete series and $[1/4, \infty)$ for the principal series.

%
%

As one might expect, the Dirac operator $D$, defined on spinor fields on $G$, provides an interesting element of the $K$-theory of $\fA$. 
 We examine the spectrum of $D$ by identifying its Fourier transform $\widehat{D}$ and in particular we investigate how the spectrum  changes as the Casimir parameter varies over $(-\infty, \infty)$. 
 
 The operator $\widehat{D}$ can be thought of as a field of operators on the tensor product of the field of representations of $G$ with the field of spinors. 
 The idea is to decompose the tensor product field into $\widehat{D}$-invariant pieces. We have one infinite dimensional field on which $\widehat{D}$ has a spectral gap and 
 thus is trivial at the level of $KK$-theory. This leaves a one-dimensional field for which the spectrum varies from minus infinity to infinity with the Casimir parameter. 
  This establishes  that $D$ gives a non-trivial element in $K$-theory, and a homotopy argument  then shows that  the Dirac class generates $KK^1(\C,\fA)$. 
 
We provide also a detailed analysis of the localised spectra of $D$ over the tempered dual of the group $G$. This demonstrates that there is much more information about the Dirac operator that can be derived from representation theory of $G$ than that provided by $K$-theory, see Figures \ref{fig:spectra1}, \ref{fig:spectra2}, and \ref{fig:helix}. We conclude by computing the Dirac cohomology of D in the sense of \cite{HP}.

\section{Representation theory and the Fourier transform}\label{reduced C^*-algebra}     
We begin with the Plancherel formula of  Puk\'anszky \cite{P} for the universal cover of $\SL_2(\R)$.

\begin{thm}\label{PT}  The following representations enter into the Plancherel formula: 
 \[
 \begin{array}{lll}
 \mathrm{Principal\; series}: &  \{(V_{q, \tau},\pi_{q,\tau}):   q \geq 1/4,  \: 0 \leq \tau \le 1\},   &  \Omega = q\\
\mathrm{Discrete\; series}: &   \{ (\ds_{\ell,\pm},\omega_{\ell,\pm}): \; \ell > 1/2\},  &  \Omega = \ell(1 - \ell)
\end{array}
\]
where $\Omega$ is the Casimir operator.    For every test function $f$ on $G$,  smooth with compact support, we have

\[
f(e) = \int_0^{\infty} \int_0^1 \sigma[ \Re \tanh \pi(\sigma + i\tau)] \Theta(\sigma,\tau)(f)d\tau d\sigma + \int_{1/2}^{\infty}(\ell - 1/2)\Theta({\ell})(f)d\ell
\]
where the Harish-Chandra characters are 
\[
\begin{split}
\Theta(\sigma,\tau)(f) & =  \trace \int_G \pi_{q,\tau}(g) f(g)dg\\
\Theta(\ell)(f) & =  \trace \int_G (\omega_{\ell,+} \oplus \omega_{\ell,-})(g)f(g)dg
\end{split}
\]
and  $\sigma = \sqrt{q - 1/4}$.  
\end{thm}

This is a \emph{measure-theoretic} statement.  We need a more precise statement in topology.

As $\ell\to1/2$ the discrete series representations $(\ds_{\ell,\pm},\omega_{\ell,\pm})$ tend to limits $(\ds_{1/2,\pm},\omega_{1/2,\pm})$, which are not strictly speaking elements of the discrete series. The representation $(V_{1/4,1/2},\pi_{1/4,1/2})$ in the principal series, which we call the \emph{limit-of-discrete-series} for $G$, is the direct sum of the two representations $(\ds_{1/2,+},\omega_{1/2,+})$ and $(\ds_{1/2,-},\omega_{1/2,-})$, see Eqn.(2.4) in \cite[p.40]{KM}.  (Note that all other elements of the principal series are irreducible.)


We will define the \emph{parameter space} $\cZ$ to be the union of the sets
\[
\{q\in \R: q\leq 1/4\}
\]
\[
\{(q,\tau)\in \R\times \R : q\geq 1/4, 0\leq \tau\leq 1\}
\]
with identification of the point $1/4$ in the first set with $(1/4,1/2)$ in the second, and with identification of  $(q,0)$ with $(q,1)$ for all $q\geq 1/4$.

The $G$-Hilbert spaces $V_{q,\tau}$ form a continuous field of Hilbert-spaces over $q\geq 1/4, 0\leq \tau\leq 1$. We extend this to a continuous field $V_*$ of $G$-Hilbert-spaces over $\cZ$ by defining
\[
V_q=\ds_{\ell,+}\oplus \ds_{\ell,-}
\]
where $q=\ell(1-\ell)$ noting that
\[
V_{1/4} = V_{1/4,1/2}, \quad V_{q,0} = V_{q,1} \quad \quad \forall q \geq 1/4.
\]
The Casimir operator on the $G$-modules $V_q$ and $V_{q,\tau}$ is precisely the multiplication by the parameter $q$.

We have the following structure theorem.

\begin{thm}\label{Fourier transform of C^*_r(G)}   Let  $\fA$ denote the reduced $C^*$-algebra $C^*_r(G)$.   The Fourier transform $f \mapsto \widehat{f}$ induces an isomorphism of 
$\fA$  onto the $C^*$-algebra
\[
\Ahat:=\{F \in C_0(\cZ,\fK(V_*)) : F(q)\ds_{\ell,+} \subset \ds_{\ell,+},  F(q)\ds_{\ell,-}\subset \ds_{\ell,-} 
\;\mathrm{if} \; q \leq 1/4\}.
\]
\end{thm}

\begin{proof}
In \cite{KM} Kraljevi\'{c} and Mili\v{c}i\'{c} gave a full description of the Fourier transform of the full $C^*$-algebra $C^*_{max}(G)$. Paraphrasing their work the algebra can be identified with an algebra of compact operator functions on a space  $\cZ_{max}$.

\begin{figure}[htbp] 
   \centering
   \includegraphics[width=3in]{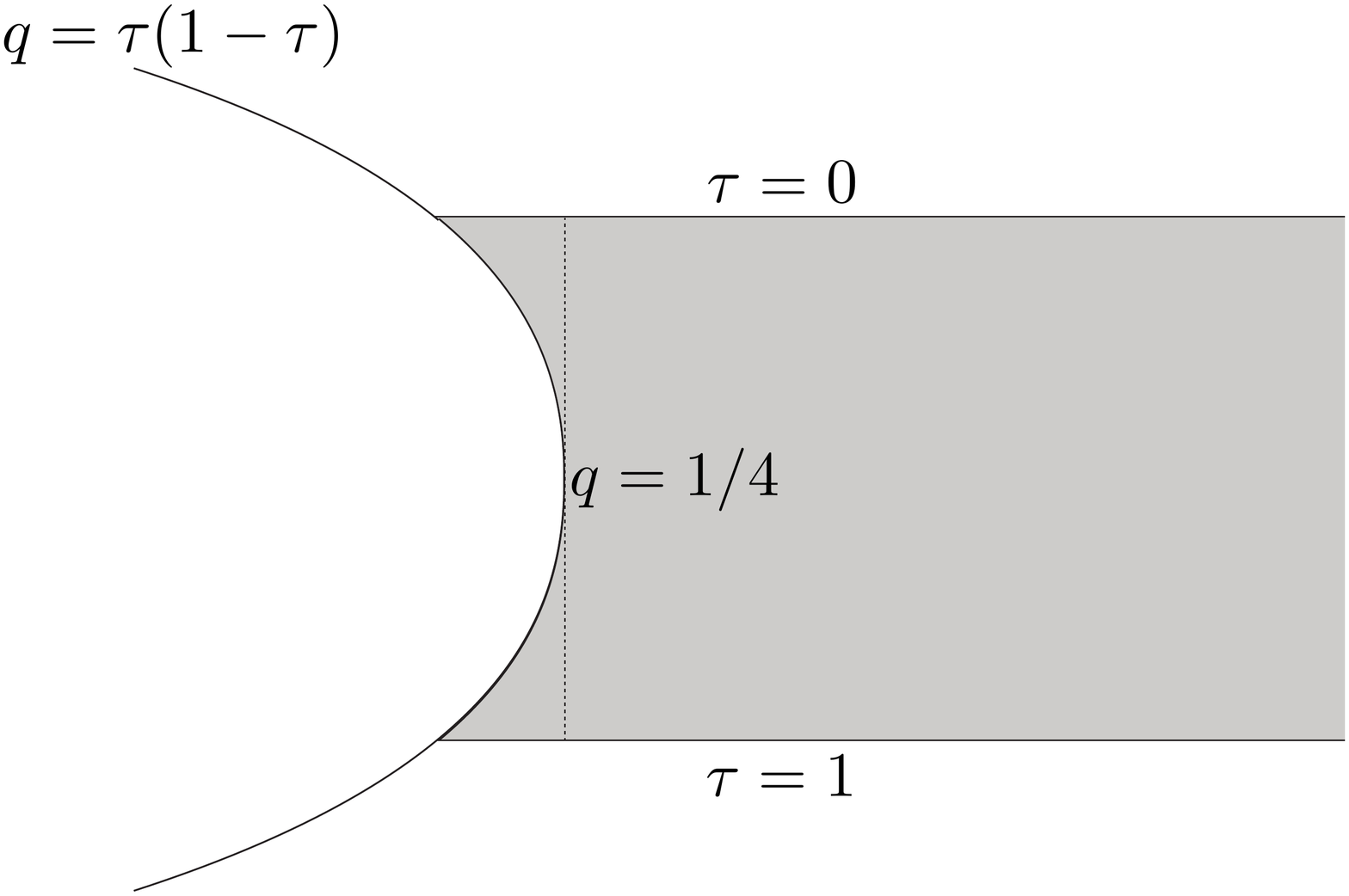} 
   \caption{The parameter space $\cZ_{max}$.}
   \label{fig:example}
\end{figure}
 

The conditions that the compact operator valued functions are required to satisfy capture all the equivalences among the unitary representations of $G$, see \cite[p. 40]{KM}. The equivalencies all occur around the boundary of the shaded region. In particular:
\begin{itemize}
\item there is a unitary identification of the operators along the line $\tau=0$ with those along $\tau=1$;
\item the operators are required to respect a decomposition of the Hilbert space on the parabola $q=\tau(1-\tau)$ for $0<\tau<1$ arising from the fact that these points of the parameter space correspond to pairs of inequivalent irreducible unitary representations;
\item at $(0,0)$ and $(0,1)$ the operators respect a decomposition into $3$ subspaces corresponding to three inequivalent irreducible unitary representations, one of which is trivial.

\end{itemize}

 
 
 For the reduced $C^*$-algebra we need to take those points of the parameter space and those subspaces of the Hilbert space which correspond to the support of the Plancherel measure, which we identify using Theorem \ref{PT}. The complementary series has Plancherel measure $0$ hence we only get points with $q\geq 1/4$ or $q=\tau(1-\tau)$. Moreover, along the parabola we now get a single irreducible representation $(\ds_{\tau,+},\omega_{\tau,+})$ for $\tau>1/2$ and $(\ds_{1-\tau,-},\omega_{1-\tau,-})$ for $\tau<1/2$ , while at the  limit-of-discrete-series, with parameter $(1/4,1/2)$, we obtain the direct sum of  $(\ds_{1/2,+},\omega_{1/2,+})$ with $(\ds_{1/2,-},\omega_{1/2,-})$.
\begin{figure}[h] 
   \centering
   \includegraphics[width=3in]{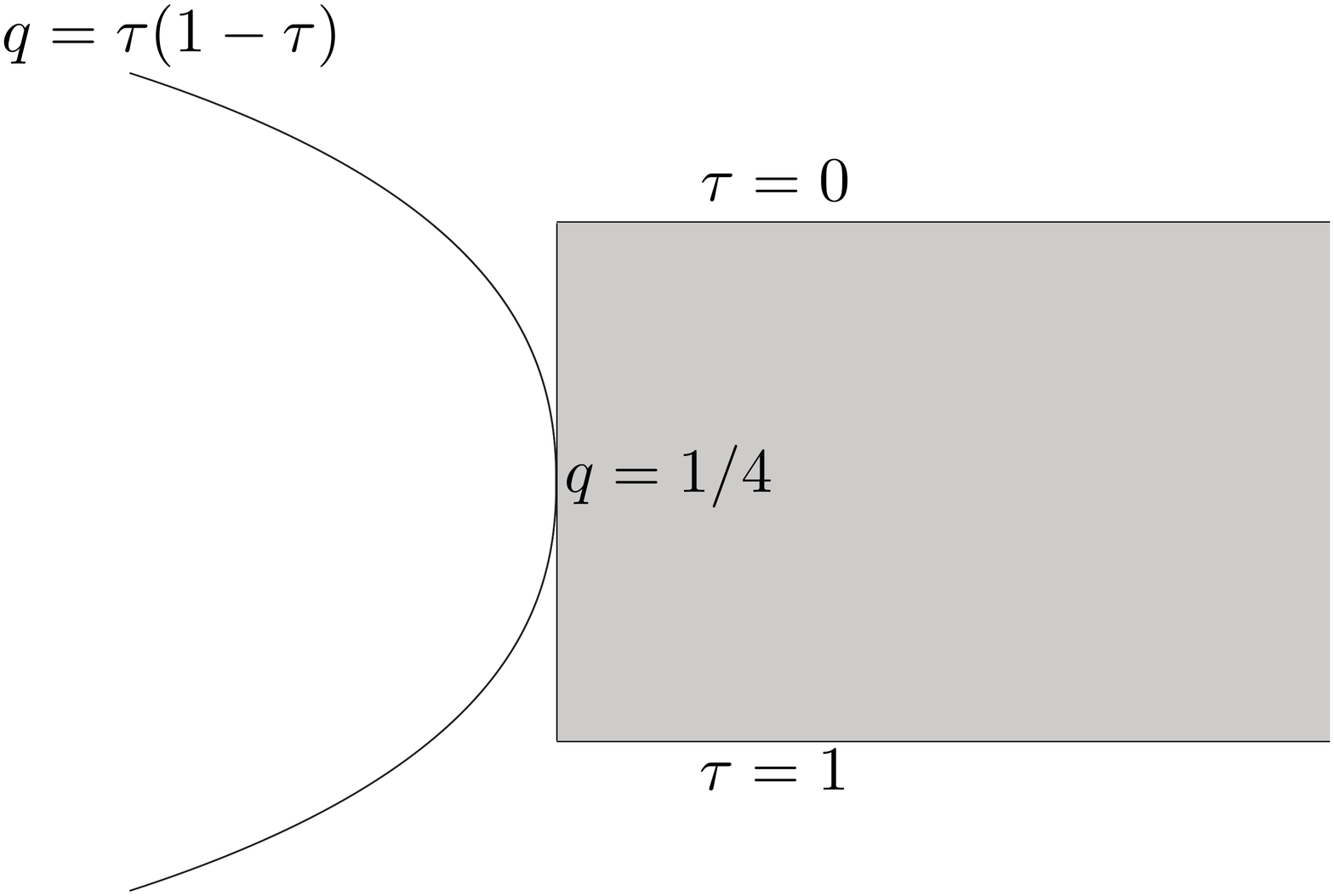} 
   \caption{The parameter space $\cZ_{red}$.}
   \label{fig:example2}
\end{figure}

 This identifies the Fourier transform of the reduced $C^*$-algebra with an algebra of compact operator valued functions on the space $\cZ_{red}$, satisfying:
 
 \begin{itemize}
\item As before, there is a unitary identification of the operators along the line $\tau=0$ with those along $\tau=1$;
 \item along the branch of the parabola $\tau>1/2$ the compact operators have range in  $\ds_{\tau,+}$, while  along the branch $\tau<1/2$ the operators have range in $\ds_{\tau,-}$;
 \item at $\tau=1/2$ the operators respect the decomposition $\ds_{1/2,+}\oplus \ds_{1/2,-}$.
 \end{itemize}

 \begin{figure}[h] 
   \centering
   \includegraphics[width=4.7in]{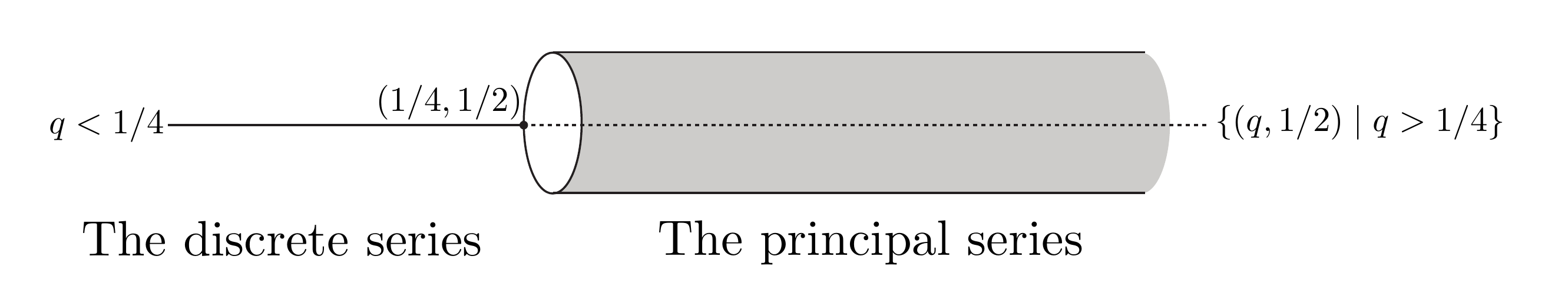} 
   \caption{The parameter space $\cZ$. Note that $q$ is the parameter defining the Casimir operator and the marked point is the limit-of-discrete-series.}
   \label{fig:example2}
\end{figure}

 The parameter space $\cZ$ is obtained from $\cZ_{red}$ by identifying the two branches of the parabola and by identifying the line $\tau=0$ with the line $\tau=1$. The first identification allows us to view the pairs of operators as a single operator respecting the decomposition as required by the theorem, while the second identification corresponds to the unitary equivalence of the representations on these two lines. This yields the required isomorphism.
\end{proof}



   Note that the Jacobson topology on the primitive ideal spectrum of $\fA$ is exactly right: it has a double point at $q = 1/4 \in \cZ$, where the  unitary representation $\pi_{1/4,1/2}$ is reducible and is Hausdorff away from this point.

We observe that the algebra $\Ahat$ is (strongly) Morita equivalent to the algebra
\[
\fB:=\{F \in C_0(\cZ,M_2(\C)) : F(q)\text{ is diagonal if} \; q \leq 1/4\}.
\]
One way to see this is to note that $\Ahat\cong \fB\otimes \fK$ however it is instructive to consider the explicit bimodules yielding the Morita equivalence. 

We introduce the notation $\C^{2}_+$ and $\C^{2}_-$ for the subspaces $\C\oplus 0$ and $0\oplus \C$ in $\C^2$. We now form the submodule $\cE$ of 
$C_0(\cZ,V_*\otimes \C^2)$ consisting of all functions $F$ such that for $q=\ell(1-\ell)\leq 1/4$, we have 
\[
F(q)\in (\ds_{\ell,+}\otimes \C^{2}_+) \oplus (\ds_{\ell,-}\otimes \C^{2}_-).
\]

The module $\cE$ can be equipped with two (pointwise) inner products. Firstly we have a $\fB$-valued inner product, which is to say a pointwise $M_2(\C)$-valued inner product satisfying the required diagonality condition. For $F=F_1\otimes F_2$, $G=G_1\otimes G_2$ the inner product is defined to be 
\[
\langle F,G\rangle_{\fB}=\langle F_1(z),G_1(z)\rangle_{V_z} \,F_2(z)\langle G_2(z),-\, \rangle_{\C^2}.
\]
Secondly we have an $\Ahat$ valued inner product, which pointwise takes values in $\fK(E)$. For $F=F_1\otimes F_2$, $G=G_1\otimes G_2$ the inner product is defined to be 
\[
\langle F,G\rangle_{\Ahat}=\langle F_2(z),G_2(z)\rangle_{\C^2} \,F_1(z)\langle G_1(z),-\, \rangle_{V_z}.
\]
The Hilbert modules obtained by equipping $\cE$ with these two inner products effect a Morita equivalence between the algebras $\Ahat$ and $\fB$.

We remark that a field of operators on the field $V_*$ of Hilbert spaces can naturally be regarded as an adjointable operator on $\cE$ with the $\fB$-valued inner product.

\section{The Dirac operator}\label{the K-cycle}   

Let $\fg$ denote the Lie algebra $\fsl_2(\R)$, let $U(\fg)$ denote the universal enveloping algebra of $\fg$, and let $C(\fg)$ denote the Clifford algebra of $\fg$ with respect to the negative definite quadratic form on $\fg$.   Let $X_0, X_1, X_2$ denote an orthonormal basis in $\fg$.   Note that the notation in 
\cite[(1.1)]{P} is $l_k = X_k$.   

Following the algebraic approach in \cite[Def. 3.1.2]{HP} the \emph{Dirac operator} is the element of the algebra $U(\fg) \otimes C(\fg)$ given by
\begin{align*}
D  = X_0 \otimes c(X_0)  + X_1 \otimes c(X_1) + X_2 \otimes c(X_2)
\end{align*}
where $c(X_k)$ denotes Clifford multiplication by $X_k$.

    Let
\[
\sigma_0 = \begin{pmatrix}
1 & 0\\
0 & -1
\end{pmatrix}, \quad \sigma_1 = \begin{pmatrix}
0 & 1\\
1 & 0
\end{pmatrix}, \quad \sigma_2 = \begin{pmatrix}
0 & i\\
-i & 0
\end{pmatrix}
\]
and set
\[
c(X_k) = i\sigma_k, \quad \quad \quad k = 0,1,2
\]
Then we have 
\[
c(X_k)^2  = -1 
\]
for all $k = 0,1,2$.

We have 

\begin{align*}\label{Dirac}
D & = i(X_0 \sigma_0  + X_1 \sigma_1  + X_2 \sigma_2 )\\
& = i\left(\begin{array}{cc}X_0 & X_1 + iX_2\\ X_1 - iX_2 & -X_0 \end{array} \right)
\end{align*}

The operator $D$ acts on 2-spinor fields in the following way: the elements of the Lie algebra $\fg$ give rise to \emph{right}-invariant vector fields on $G$ and in this way $X_j,j=0,1,2$ form differential operators on scalar fields. The matrix then acts by differentiating the components of a spinor field.

Viewing a compactly supported 2-spinor field as a pair of scalar valued functions on the group $G$, it is an element of $C_c(G)\oplus C_c(G)\subseteq \fA\oplus \fA$. In this way the Dirac operator $D$ gives rise to an unbounded adjointable operator on $\fA\oplus \fA$ viewed as a Hilbert module over $\fA$.

Let now $\pi$ be a unitary representation, in the principal series or the discrete series of $G$, on a Hilbert space $V_{\pi}$.   The infinitesimal generators $H_0, H_1, H_2$, which act on the Hilbert space $V_{\pi}$,  are determined by the following equation \cite[p.98]{P}:
\[
\exp( - itH_k) = \pi(\exp(tX_k)) \quad \quad \quad \forall t \in \R, k = 0,1,2.
\]

Let us now recall the parameter space $\cZ$ from $\S$\ref{reduced C^*-algebra} and corresponding continuous field of $G$-Hilbert-spaces over $\cZ$:
\[
\{V_q : q \leq 1/4\}, \qquad
\{V_{q,\tau} : q \geq 1/4, 0 \leq \tau \leq 1\}.
\]

On each of the Hilbert spaces $V_q$ and $V_{q,\tau}$ we therefore have three self-adjoint operators, namely $H_0$, $H_1$ and $H_2$.  These form a field of operators on the field of Hilbert spaces $V_*$.     The spectrum of $H_0$ is discrete with eigenvalues $m = \ell, \ell+1,\ell+2,\dots$ and $ m = -\ell,\-\ell-1,-\ell-2,\dots$ in the case that $q<1/4$ and with eigenvalues $m \in \tau+\Z$ for $q\geq 1/4$. Each eigenvalue has multiplicity $1$ and we let $f_m$ be an orthogonal basis of eigenvectors of $H_0$ so that
\[H_0f_m = mf_m.\]

\medskip

Following \cite[p.100]{P}, we define 
\[
H_+ = H_1 + iH_2, \quad \quad H_- = H_1 - i H_2
\]

In addition, we have the following equations
\[
\begin{split}
H_+  f_m & = (q+m(m+1))^{1/2}f_{m+1}\\
H_- f_m & = (q+m(m-1))^{1/2}f_{m-1}
\end{split}
\]
which hold for all $m$ when $q\geq 1/4$ and where the first equation holds for all $m\neq -\ell$, the second for all $m\neq \ell$ when $q<1/4$. The special cases of $H_+ f_{-\ell}$ and $H_- f_\ell$ are both zero.

By analogy with the Dirac operator $D$ above, we construct a field of self-adjoint operators
\begin{align}
\Hdirac=
\begin{pmatrix}
H_0 & H_1 + iH_2\\
H_1 - iH_2  & - H_0
\end{pmatrix}
\end{align}
on the field of Hilbert spaces $V_*\oplus V_*$. Since the algebra $\Ahat$ consists of fields of compact operators on $V_*$, the operator $\Hdirac$ can also be thought of as acting on $\Ahat\oplus\Ahat$ by composition. (One must additionally note that the operators $H_j,j=0,1,2,$ respect the decomposition of $V_q$ as $\ds_{\ell,+}\oplus \ds_{\ell,-}$ for $q\leq 1/4$.)

Since the operator $D$ acts on $\fA\oplus \fA$ and we have an isomorphism from $\fA$ to $\Ahat$ given by the Fourier transform (Theorem \ref{Fourier transform of C^*_r(G)}), we obtain an operator $\hatD$ on $\Ahat\oplus\Ahat$. We will show that
$$\hatD=-\Hdirac.$$ 
It suffices to show that the isomorphism $\fA\cong \Ahat$ takes the differential operator $iX_j$ to $-H_j$ for each $j$.

For $f$ a smooth compactly supported function on $G$ the Fourier transform of $f$ is defined by
$$\hat f(\pi)=\int f(g)\pi(g)\,dg.$$
Since $H_j$ is the infinitesimal generator of the 1-parameter group $\exp( - itH_j) = \pi(\exp(tX_j))$ we have $H_j=i\frac d{dt}\pi(\exp(tX_j))|_{t=0}$. Hence
\begin{align*}
-H_j\hat f(\pi)&=-\int f(g)H_j\pi(g)\,dg\\
&=-\int f(g)i\frac d{dt}\pi(\exp(tX_j))|_{t=0}\pi(g)\,dg\\
&=-i\frac d{dt}\int f(g)\pi(\exp(tX_j)g)\,dg|_{t=0}\\
&=-i\frac d{dt}\int f(\exp(-tX_j)g')\pi(g')\,dg'|_{t=0} \quad\text{by left invariance of $dg$}\\
&=-i\int \frac d{dt}f(\exp(-tX_j)g')|_{t=0}\pi(g')\,dg'\\
&=i\int X_j(f)\pi(g')\,dg'\\
&=i\widehat{X_j(f)}(\pi)
\end{align*}
since $X_j$ is a right-invariant vector field. This establishes that the Fourier transform takes $iX_j$ to $-H_j$ and hence that $\hatD=-\Hdirac$.

This exhibits $\hatD$ as  a field of self adjoint operators on our parameter space $\cZ$ as expected from the $G$-invariance of $D$. The spectrum of $\hatD$ is the union of the spectra of the local operators. Each local operator has a discrete spectrum which we will  describe in some detail.

Recall that $\Hdirac$ is defined on the field $V_*\oplus V_*$, and a crucial observation at this point is the emergence of two dimensional invariant subspaces for $\Hdirac$ at each point of the parameter space $\cZ$.

Each such subspace $E_m$ is spanned by a pair of vectors
$\begin{pmatrix} f_m \\ 0 \end{pmatrix}$ and $\begin{pmatrix} 0 \\  f_{m-1} \end{pmatrix}$,
where $m$ is in the set $\tau + \Z$ for $q \geq 1/4$ and $m =  \ell +1, \ell +2, \ldots$ or $- \ell, -\ell - 1, \ldots$ for $ q < 1/4$. We have the following equations
\[
\Hdirac\begin{pmatrix} f_m \\ 0 \end{pmatrix}=
\begin{pmatrix}
H_0 & H_+\\
H_- & - H_0
\end{pmatrix}
\begin{pmatrix} f_m \\ 0 \end{pmatrix}
= 
\begin{pmatrix}
mf_m \\
(q+m(m-1))^{1/2}f_{m-1}
\end{pmatrix}
\]
and
\[
\Hdirac
\begin{pmatrix} 0 \\  f_{m-1} \end{pmatrix}
=
\begin{pmatrix}
H_0 & H_+\\
H_- & - H_0
\end{pmatrix}
\begin{pmatrix} 0 \\  f_{m-1} \end{pmatrix}
=
\begin{pmatrix}
(q+m(m-1))^{1/2}f_{m}\\
-(m-1)f_{m-1}
\end{pmatrix}.
\]
With respect to this basis, the operator $\Hdirac$
is given by the following symmetric matrix 
\begin{equation}\label{thematrix}
\begin{pmatrix}
m & (q+m(m-1))^{1/2} \\
(q+m(m-1))^{1/2} & -(m-1)
\end{pmatrix}. 
\end{equation}

This symmetric matrix has the following eigenvalues 
\begin{equation}\label{eigenvalues}
\lambda = \frac{1}{2}\pm \sqrt{1/4+q +2m(m-1)}. 
\end{equation}

In the case that $q \geq 1/4$ the subspaces $E_m$ for $m\in\tau+\Z$ span the whole of $V_{q,\tau}$.    However, for $ q < 1/4$ there are a further two $1$-dimensional subspaces spanned by the vectors 
\[\begin{pmatrix}
f_\ell \\ 0 
\end{pmatrix}, \begin{pmatrix}
0 \\ f_{-\ell}
\end{pmatrix}
\]
These subspaces are invariant since $H_-(f_{\ell}) = 0$ and $H_+(f_{-\ell}) = 0$, and moreover, they span a $2$-dimensional eigenspace with eigenvalue $\ell$.

Now consider the case $q=1/4, \tau=1/2$. Taking $m=1/2$ the matrix in (\ref{thematrix}) is $\frac{1}{2}I$ hence $\Hdirac$ has a $2$-dimensional eigenspace with eigenvalue $1/2$. This matches up with the above $2$-dimensional eigenspace as $q\rightarrow 1/4^-$ and correspondingly $\ell\rightarrow 1/2^+$, reflecting the fact that $\pi_{1/4,1/2}$ is the limit-of-discrete-series representation of $G$.

%

%
%
%
\begin{defn}
Let $G$ be a second countable locally compact group and $\mathcal{E}$ a $C^*_r(G)$-Hilbert module.  Let $(V,\pi)$ be a representation whose irreducible constituents are in the reduced dual of $G$. For $T$ an adjointable operator on $\mathcal{E}$, the \emph{localisation of $T$ at $\pi$} is the operator $T\otimes 1$ on $\mathcal{E}\otimes_{C^*_r(G)}V$.  The \emph{localisation of the spectrum of $T$ at $\pi$} is the spectrum of the localisation of $T$.
\end{defn}

The localisation of our Dirac operator $D$ at the representation $\pi_{q,\tau}$ corresponds to taking the fibre of $\hatD$ at the point $(q,\tau)$ in $\cZ$. On the other hand the fibre of $\hatD$ at a point $q=\ell(1-\ell)<1/4$ in $\cZ$ is the direct sum of the localisations of $D$ at $\omega_{\ell,\pm}$.

Equation (\ref{eigenvalues}), for the appropriate values of $m$, describes most of the spectrum for the fibre of $\Hdirac$. The exception is provided by a further $2$-dimensional eigenspace with eigenvalue $\ell$ when $q<1/4$. Indeed for these values of $q$ each eigenvalue given by Equation \ref{eigenvalues} also has multiplicity $2$: the values $m=(\ell+k)$ and $m=\ell-(k-1)$ (for $k= 1,2, \ldots$) yielding the same eigenvalue where the first occurs in the localisation over $\omega_{\ell,+}$ and the second over $\omega_{\ell,-}$ .

As $\hatD=-\Hdirac$ we  thus obtain the following description of the localised spectra of $D$.

\begin{figure}[h] 
   \centering
   \includegraphics[width=1.5in]{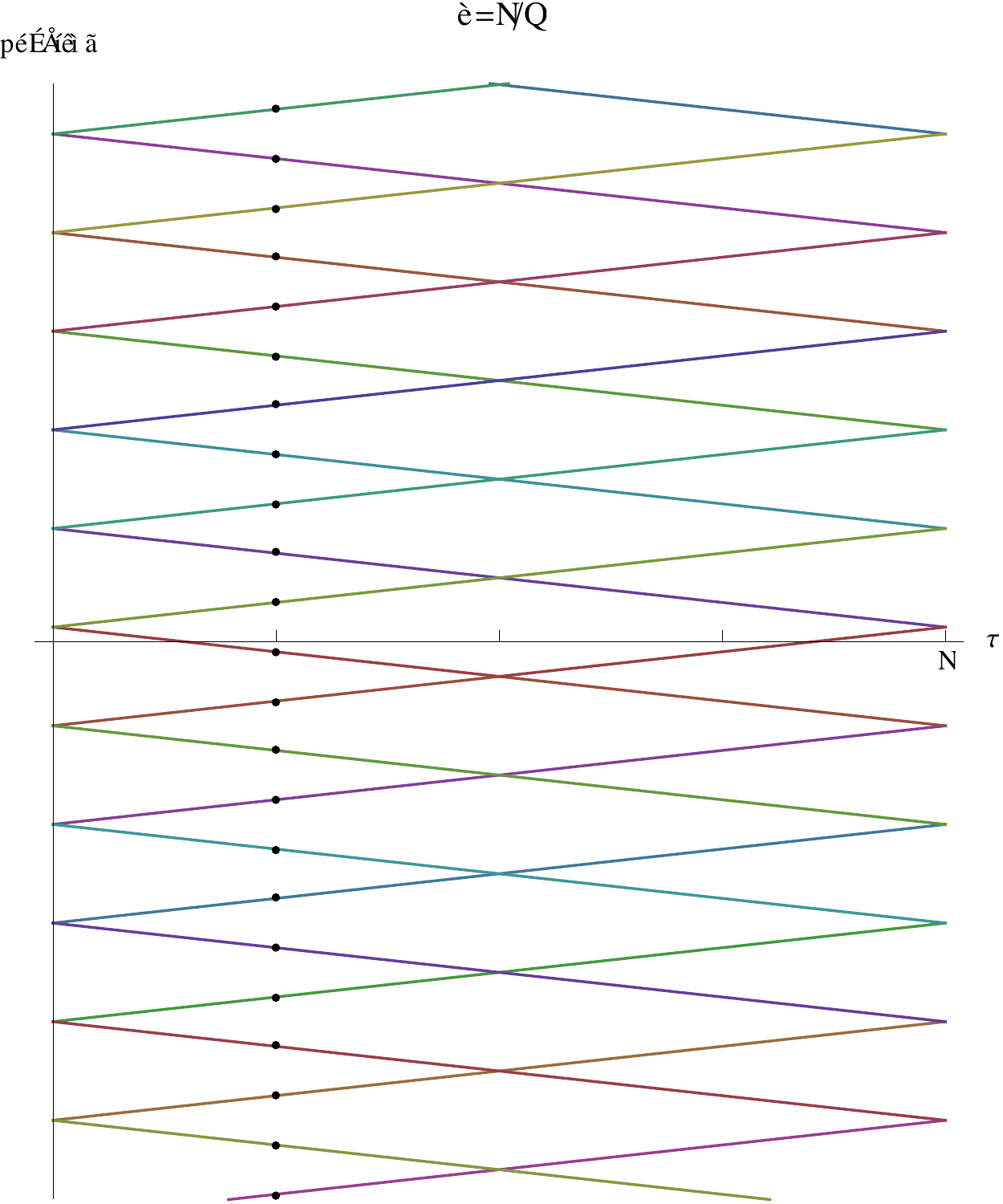} \includegraphics[width=1.5in]{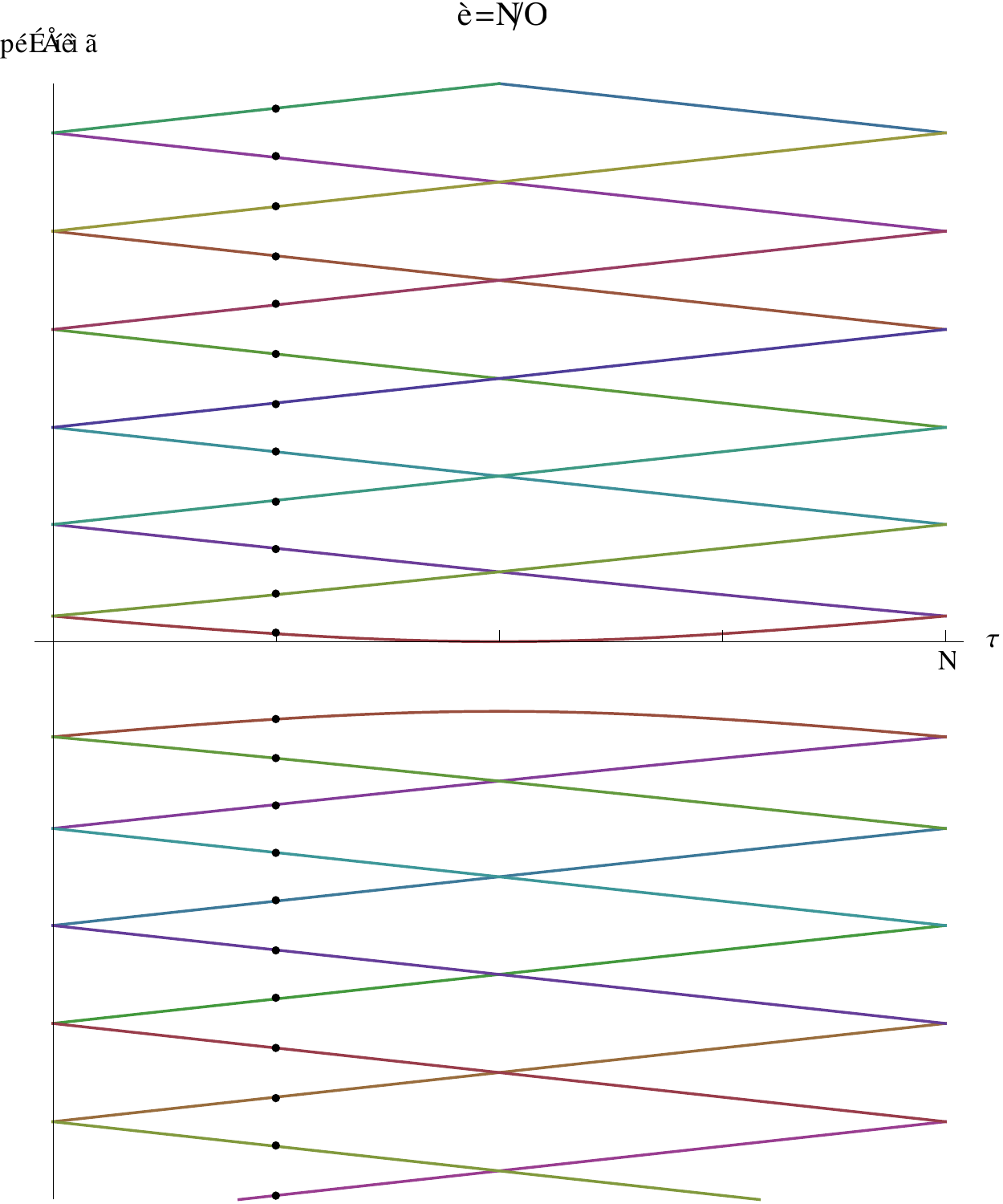} \includegraphics[width=1.5in]{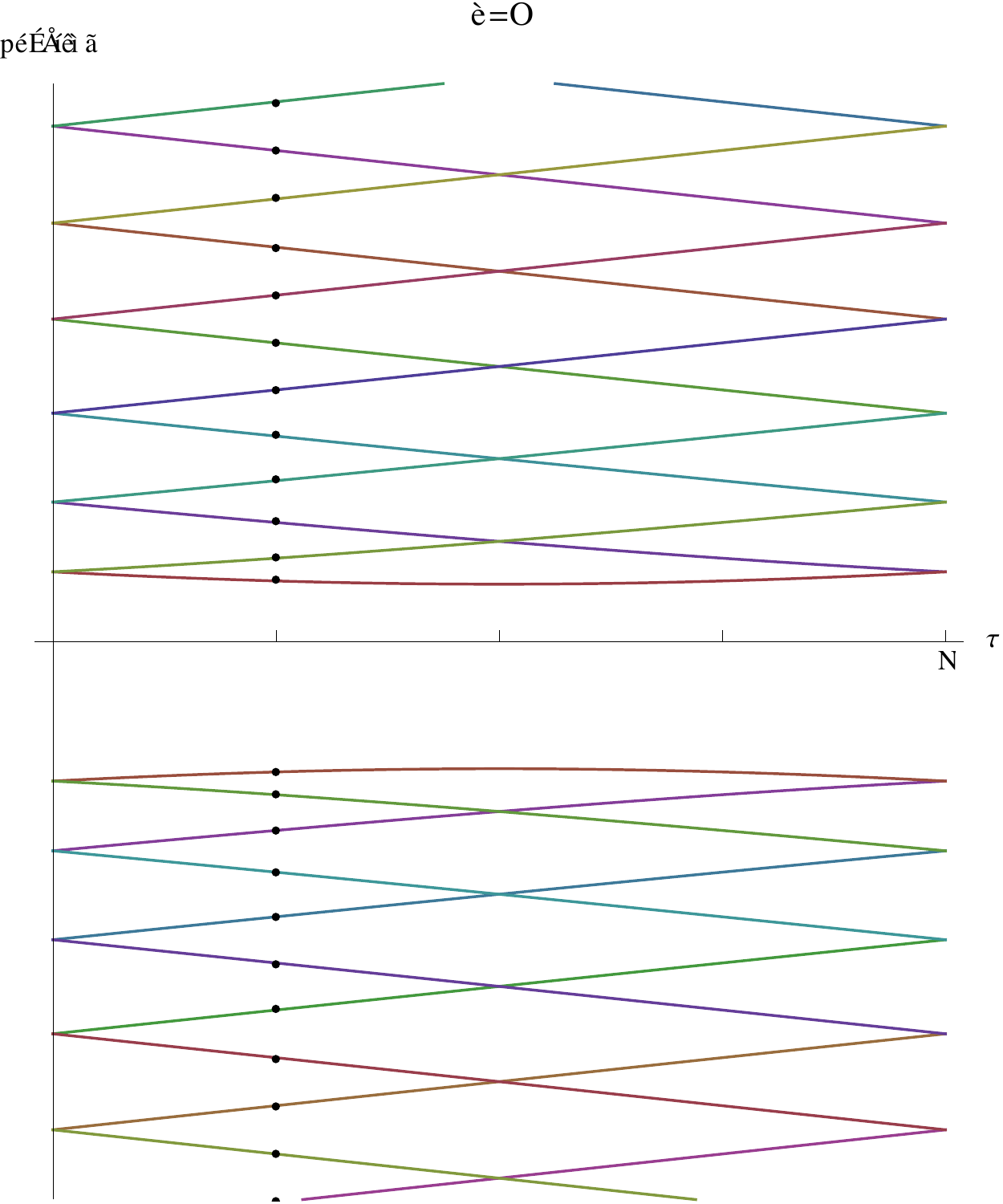} 
   \caption{The localised spectra of $D$ over the circles of representations  $\pi_{q,\tau}$ at $q=1/4,1/2,2$, and corresponding to $k= -6,  \ldots, 6$.}
   \label{fig:spectra1}
\end{figure}
\begin{thm}\label{Theorem_3.2}
The spectrum of the Dirac operator $D$ localised at the representation $\pi_{q,\tau}$ is 
\[
\left\{-\frac{1}{2}\pm \sqrt{1/4+q +2(\tau+k)(\tau+k-1)}\, :\, k\in \Z\right\}.
\]
The spectrum of the Dirac operator localised at $\omega_{\ell,+}$ agrees with that at  $\omega_{\ell,-}$  and is given by the formula 
\[
\{-\ell\} \cup\left\{-\frac{1}{2}\pm \sqrt{1/4+\ell(1-\ell) +2(\ell+k)(\ell+k-1)}\, :\, k\ = 1,2,\dots \right\}.
\]


\end{thm}

Note  that in the principal series ($q\geq 1/4$) the eigenvalues generically have multiplicity $1$. However for $\tau=0$ the eigenvalues have multiplicity $2$, while for $\tau=1/2$ the eigenvalues have multiplicity $2$ for $k\not=0$. The $k=0$ eigenvalues for $\tau=1/2$ are $ -1/2\pm \sqrt{q-1/4}$ and have multiplicity $1$ except when $q=1/4$, at the limit-of-discrete-series representation, where they coallesce. See  Figure \ref{fig:spectra1}.

\begin{figure}[h] 
   \centering
   \includegraphics[width=2.5in]{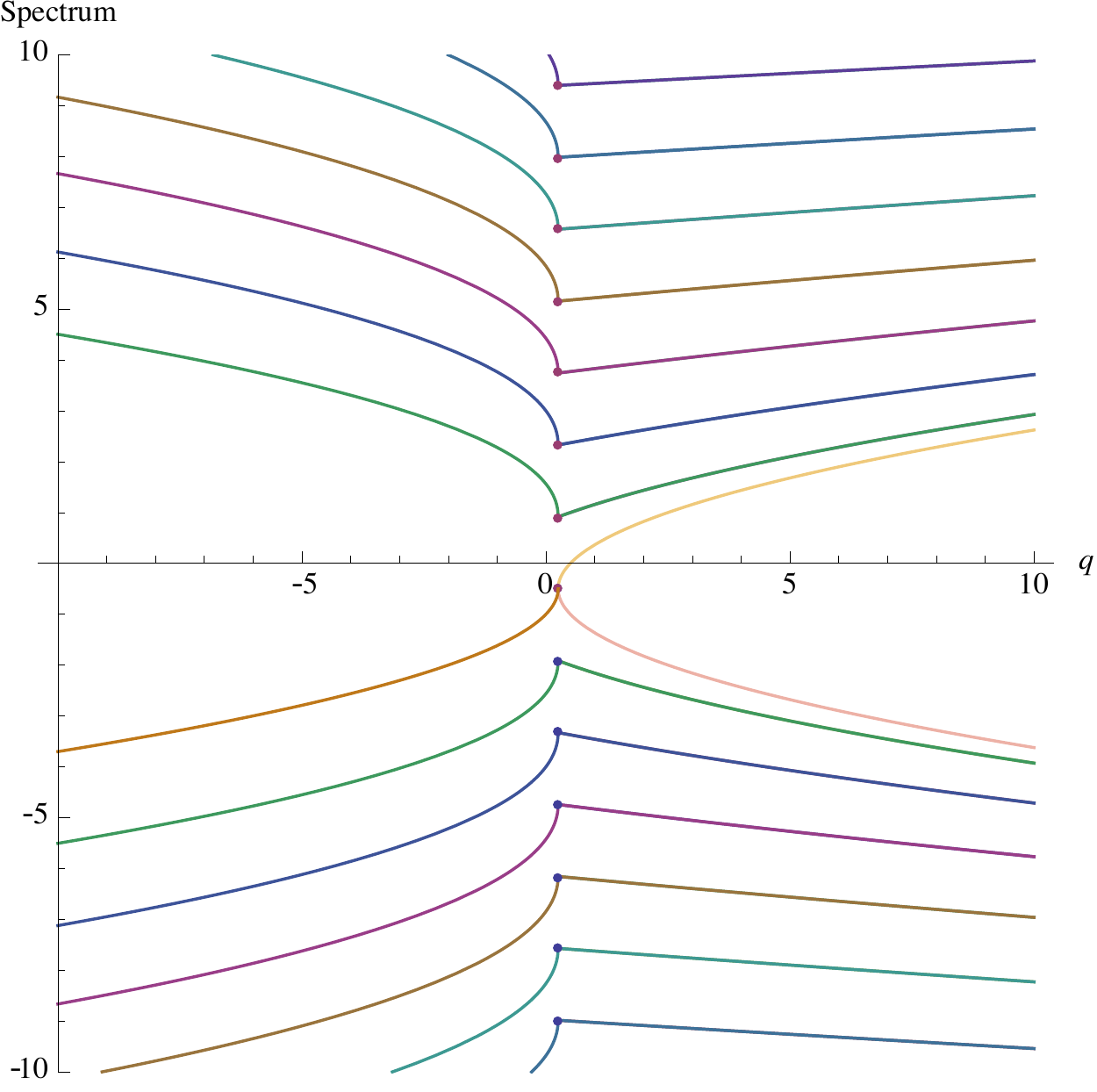} 
   \caption{The localised spectra of $D$ over the line of representations 
   \[
   \{\omega_{\ell,+}\oplus \omega_{\ell,-}: q=\ell(1-\ell)<1/4\}\cup\{\pi_{q,1/2}: q\geq 1/4\}.
   \]
   The cusps and bifurcation occur at $q=1/4$.}
   \label{fig:spectra2}
\end{figure}

Note from Figure \ref{fig:spectra2} that something very special occurs in the limit-of-discrete series when $q = 1/4$ and $\tau = \ell=1/2$.  The $2$-dimensional eigenspace with eigenvalue $-\ell$ along the branch $q<1/4$ splits at $q=1/4$ into two $1$-dimensional eigenspaces, with eigenvalues diverging to $\pm\infty$ as $q\rightarrow \infty$. Taking the upper branch allows the eigenvalue to range continuously from $-\infty$ to $\infty$ with the parameter $q$, which, as we shall see,  gives rise to the non-triviality of $D$ as a class in $K$-theory.

As noted above, the operator $\Hdirac$ can be thought of either as an operator on $\Ahat\oplus \Ahat$ or an operator on $V_*\oplus V_*$. Taking the latter view, the operator acts on the space $C_0(\cZ,V_*\oplus V_*)$ of continuous sections of $V_*\oplus V_*$ vanishing at infinity, which is a Hilbert module over $C_0(\cZ)$. Indeed there is an isomorphism of Hilbert modules
$$(\Ahat\oplus \Ahat)\otimes_{\Ahat}C_0(\cZ,V_*) \cong C_0(\cZ,V_*\oplus V_*)$$
and viewing $\Hdirac$ as an operator on $\Ahat\oplus \Ahat$, the corresponding operator on $C_0(\cZ,V_*\oplus V_*)$ is given by $\Hdirac \otimes 1$.

The algebra $\Ahat$ is contained in $C_0(\cZ,\fK(V_*))$ (or $C_0(\cZ,\fK)$ for short) and we can factorise the tensor product further as
$$(\Ahat\oplus \Ahat)\otimes_{\Ahat}C_0(\cZ,\fK)\otimes_{C_0(\cZ,\fK)}C_0(\cZ,V_*).$$
Hence viewing $\Hdirac$ as an operator on sections of $V_*\oplus V_*$ corresponds to the composition of the forgetful inclusion of $\Ahat$ into $C_0(\cZ,\fK)$ with the Morita equivalence from $C_0(\cZ,\fK)$ to $C_0(\cZ)$ (which is implemented by the module $C_0(\cZ,V_*)$).

We now consider the restriction of the field $V_*$ of Hilbert spaces to the copy of the real line inside $\cZ$ given by the union
\[
\{q\in \R: q\leq 1/4\}\cup \{(q,1/2)\in \R\times \R : q\geq 1/4\},
\]
which we regard as parametrized by $q$. Correspondingly we restrict $\Hdirac$ to the Hilbert spaces over this line. Passing from $\Hdirac$ viewed as an operator on an $\Ahat$-Hilbert module to $\Hdirac|_\R$ viewed as an operator on a $C_0(\R)$-Hilbert module corresponds to the composition
\begin{equation*}
\Ahat \hookrightarrow C_0(\cZ,\fK) 
\underset{\scriptstyle\text{Morita}}{\sim} C_0(\cZ)\twoheadrightarrow C_0(\R).
\end{equation*}
We will see in Section \ref{K-theory} that this composition induces an isomorphism at the level of $KK$-theory. In this section we will identify the class $[C_0(\R,V_*|_\R),1,\Hdirac|_\R]$ in $KK^1(\C,C_0(\R))$.

\begin{rem}  The half-line $\{ q \in \R : q \geq 1/4\}$ has the following significance in representation theory.   The corresponding unitary representations $(V_{q,1/2},\pi_{q,1/2})$ all factor through $\SL_2(\R)$ and constitute the odd principal series $\pi_q$ of $\SL_2(\R)$.   In particular, the representation $\pi_{1/4,1/2}$ descends to the limit-of-discrete series for $\SL_2(\R)$.   Furthermore it is the direct sum of two irreducible representations whose characters $\theta_+$ and $\theta_-$ do not vanish on the elliptic set, and in this respect, they resemble representations in the discrete series.   So the term \emph{limit-of-discrete-series} for $\pi_{1/4,1/2}$ is surely apt. 
\end{rem}

We  can now construct a family of complex hermitian line bundles over $\R$
\[
\{L^{(k)\pm} \mid k\in \Z\}
\]
\noindent by glueing together selected one-dimensional eigenspaces of $\Hdirac$. 

The bifurcation of the spectrum at $q=1/4$, illustrated in Figure \ref{fig:spectra2}, plays a special role and we begin with the corresponding eigenspaces. Take $q\geq 1/4$ and consider the subspace $E_{1/2}$. The restriction $\Hdirac|_{E_{1/2}}$ is given by the matrix
\[
\begin{pmatrix}
1/2 & \sqrt{q-1/4} \\
\sqrt{q-1/4} & 1/2
\end{pmatrix}
\]
from which we readily see that the vector $\begin{pmatrix} f_{1/2}\\-f_{-1/2}\end{pmatrix}$ is an eigenvector of $\Hdirac$ with eigenvalue $1/2-\sqrt{q-1/4}$. Note that the eigenvalue tends to $1/2$ as $q\to 1/4^+$.

Now for $q<1/4$ we see that $\begin{pmatrix} f_{\ell}\\-f_{-\ell}\end{pmatrix}$ is an eigenvector of $\Hdirac$ with eigenvalue $\ell=1/2+\sqrt{1/4-q}$ tending to $1/2$ as $q\to 1/4^-$.

We can thus define the $\Hdirac$-invariant complex hermitian line bundle $L^{(0)-}$ as follows.

\[ 
\text{Define} \; L^{(0)-}_q : = \begin{cases}
\C\begin{pmatrix}
f_{\ell} \\ - f_{-\ell}
\end{pmatrix} & \text{for}\; q \leq 1/4, q = \ell(1 - \ell)\\\\
\C\begin{pmatrix}
f_{1/2} \\ - f_{-1/2}
\end{pmatrix} & \text{for}\; q \geq 1/4\\
\end{cases}
\]

On this line-bundle the field of operators $\Hdirac$ is simply multiplication by the function

\[ 
\omega(q): = \begin{cases}
\frac12+\sqrt{1/4-q} & \text{ for }q\leq 1/4\\
\frac12-\sqrt{q-1/4} & \text{ for }q\geq 1/4.
\end{cases}
\]
In particular we note that $\omega(q)$ tends to $\pm\infty$ as $q\to\mp\infty$, passing through $0$ when $q=1/2$. The field of operators induces an operator on the Hilbert module $C_0(\R,L^{(0)-})$ of $C_0$-sections of the bundle, which is again multiplication by $\omega$. This is sufficient to establish:

\begin{thm} The unbounded Kasparov triple $[C_0(\R,L^{(0)-}),1,\omega]$ represents the generator of $KK^1(\C,C_0(\R))$.
\end{thm}

Similarly we have an $\Hdirac$-invariant line bundle $L^{(0)+}$ with
\[ 
\; L^{(0)+}_q : = \begin{cases}
\C\begin{pmatrix}
f_{\ell} \\ f_{-\ell}
\end{pmatrix} & \text{for}\; q \leq 1/4, q = \ell(1 - \ell)\\\\
\C\begin{pmatrix}
f_{1/2} \\ f_{-1/2}
\end{pmatrix} & \text{for}\; q \geq 1/4\\
\end{cases}
\]
on which the operator $\Hdirac$ is multiplication by the function
\[ 
\varepsilon(q): = \begin{cases}
\frac12+\sqrt{1/4-q} & \text{ for }q\leq 1/4\\
\frac12+\sqrt{q-1/4} & \text{ for }q\geq 1/4.
\end{cases}
\]
In this case we see that $\varepsilon(q)\to+\infty$ as $q\to\pm\infty$ from which we see that the corresponding Kasparov triple $[C_0(\R,L^{(0)+}),1,\varepsilon]$ represents the zero element of $KK^1(\C,C_0(\R))$.

We now examine how the remaining 2-dimensional subspaces $E_m$ match up at $q=1/4$.  For $q\geq 1/4$ we have $m=1/2+k$ where $k\in \Z$ and we exclude the case $k=0$ which we have already considered. 

For $k>0$ we take the 2-dimensional bundle $N^{(k)}$ whose fibres are $E_{1/2+k}$ for $q\geq 1/4$ and  $E_{\ell+k}$ for $q\leq 1/4$. These agree at $q=1/4$ since $\ell=1/2$ at this point.

For $k<0$ we take the 2-dimensional bundle $N^{(k)}$ whose fibres are $E_{1/2+k}$ for $q\geq 1/4$ and $E_{-\ell+1+k}$ for $q\leq 1/4$. We note that when $\ell=1/2$ we obtain $E_{-\ell+1+k}=E_{1/2+k}$.

Thus for each $k\neq 0$ the fibre of $N^{(k)}$ is $E_m$, where  $m$ is a continuous function of $q$ defined by

\[
m(q): = \begin{cases}
\frac12+\sqrt{1/4-q}+k & \text{ for }q\leq 1/4, k=1,2,\dots\\
\frac12-\sqrt{1/4-q}+k & \text{ for }q\leq 1/4, k=-1,-2,\dots\\
\frac12+k & \text{ for }q\geq 1/4, k\in \Z\setminus\{0\}
\end{cases}
\]

\noindent using the formula $\ell=1/2+\sqrt{1/4-q}$.

Recall that the eigenvalues of the restriction of $\Hdirac$ to $E_m$ are given by $\lambda^\pm=\frac{1}{2}\pm \sqrt{\Delta}$, where
\[
\Delta={1/4+q +2m(m-1)}. 
\]

We will now verify that the discriminant $\Delta$ is positive so that the eigenvalues are always distinct. 
Writing $2m(m-1)$ as $2(m-1/2)^2-1/2$ we see that 

\[
\Delta(k,q) = \begin{cases}
1/4-q+2(k^2+2|k|\sqrt{1/4-q}) & \text{ for }q\leq 1/4\\
q-1/4+2k^2 & \text{ for }q\geq 1/4.
\end{cases}
\]

This formula also makes sense when $k=0$ yielding the correct positive eigenvalue $\lambda^{+}(q)=\varepsilon(q)$, however we note that $\lambda^{-}(q)$ is not equal to $\omega(q)$. This difference is crucial since $\omega$ yields a non-trivial element of $K$-theory while $\lambda ^{-}$ would not.

We see that $\Delta$ is always at least $2$, and hence $\lambda^+\geq \frac 12+\sqrt{2}$ and $\lambda^-\leq \frac12-\sqrt{2}$. Thus the bundles of positive and negative eigenspaces within $N^{(k)}$ define $\Hdirac$-invariant line bundles which we denote $L^{(k)\pm}$. Moreover, for each $k$, $\lambda^+(q)\rightarrow+\infty$ as $q\to\pm\infty$ and $\lambda^-(q)\rightarrow-\infty$ as $q\to\pm\infty$.

This establishes the following result.

 \begin{thm}
 For each $k$ the Kasparov triples
 $[C_0(\R,L^{(k)\pm}),1,\lambda^{\pm}]$
 both represent the zero element of $KK^1(\C,C_0(\R))$.
 \end{thm}
 
 Finally, we have
 \begin{thm} The Kasparov triple
 $$[C_0(\R,V_*|_\R),1,\hatD|_{\R}]$$
 generates $KK^1(\C,C_0(\R)))$.
 \end{thm}
 \begin{proof} It suffices to check that $[C_0(\R,V_*|_\R),1,\Hdirac|_{\R}]$  generates $KK^1(\C,C_0(\R)))$, noting that $[C_0(\R,V_*|_\R),1,\hatD|_{\R}]$ is its inverse.
 
We have seen that the field of Hilbert spaces $V_*$ over $\R$ can be decomposed as the direct sum of $L^{(k)\pm}$ for $k\in \Z$ and that the operator $\Hdirac$ respects this decomposition.

Now restricting $\Hdirac$ to an operator on sections of
$$\bigoplus_{k\in\Z} L^{(k)+}$$
the operator acts by multiplication by $\lambda^{+} = \frac{1}{2}+\sqrt{ \Delta(k,q)}$. The above formulas for the discriminant show that $\lambda^{+}$ tends to infinity as $k\to\pm\infty$ and also as $q\to\pm\infty$. Hence the corresponding bounded operator $\Hdirac(1+\Hdirac^2)^{-1/2}$ is a compact perturbation of the identity operator $1$. The corresponding bounded Kasparov triple is thus trivial.

Similarly restricting $\Hdirac$ to an operator on sections of
$$\bigoplus_{k\in\Z\setminus\{0\}} L^{(k)-}$$
the operator acts by multiplication by $\lambda^{-} = \frac{1}{2}- \sqrt{\Delta(k,q)}$. Since $\lambda^{-}$ tends to minus infinity as $k\to\pm\infty$ and as $q\to\pm\infty$ the corresponding bounded operator $\Hdirac(1+\Hdirac^2)^{-1/2}$ is a compact perturbation of $-1$, and again the Kasparov triple is trivial.

We conclude that neither of these restrictions contributes to the $K$-theory and thus
$$[C_0(\R,V_*|_\R),1,\Hdirac]=[C_0(\R,L^{(0)-}),1,\omega]$$
which is the generator of $KK^1(\C, C_0(\R))$.   
\end{proof}

\section{Computing the $K$-theory}\label{K-theory}

In this section we compute the $K$-theory of $\fA$ and establish that the Dirac class is a generator of $KK^1$. We begin with a convenient reparametrisation of the space 
$\cZ$. While this looks like the parameterisation used in \cite{KM} it is not directly related to it, as our parameterisation corresponds to the reduced $C^*$-algebra whereas theirs corresponds to the full $C^*$-algebra of $G$.  

Let $\cY$ be the union of the unit disc
$
A = \{z \in \C : |z| \leq 1\}
$
and the interval  $B = [1,2]$. 

The coordinate change
\begin{align*}
(q,\tau)&\mapsto (q + 3/4)^{-1}e^{2\pi i(\tau-1/2)}, &&\text{ for $q\geq 1/4$, $0\leq \tau\leq 1$}\\
q&\mapsto 2-(5/4-q)^{-1}, &&\text{ for $q\leq 1/4$}
\end{align*}
respects the identifications in the construction of $\cZ$, and transforms the locally compact parameter space $\cZ$ into a dense subspace $\cU$ of the compact parameter space $\cY$. Explicitly $\cU$ is the union of the punctured disc with the half-open interval $[1,2)$, and the limit-of-discrete-series corresponds to the point $1$. 
\begin{figure}[h] 
   \centering
   \includegraphics[width=2in]{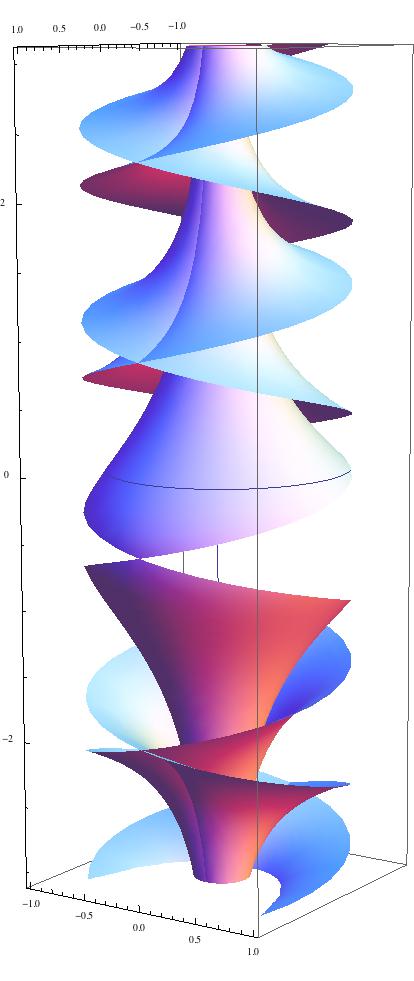}
   \caption{The spectrum of the Dirac operator localised over the principal series in $\cY$. }
   \label{fig:helix}
\end{figure}

We remark  that we can use the above parametrisation of the space $\cY$ to provide a more explicit representation of the spectrum 
of the Dirac operator localised over the principal series. Using the cylindrical coordinates  $(r,\theta,z)$ given by 
\[
r=(q+3/4)^{-1}, \theta=2\pi(\tau-1/2).
\] 
we can see that this part of the localised spectrum is given by a pair of intersecting helices, as in Figure \ref{fig:helix}.

\begin{lem}  The reduced $C^*$-algebra $\fA$ is strongly Morita equivalent to the $C^*$-algebra 
$\fD$ of all $2\times 2$-matrix-valued functions on the compact Hausdorff space $\cY$ which are diagonal on $B$, and vanish at $0$ and $2$.
\end{lem}

\begin{proof}
By Theorem \ref{Fourier transform of C^*_r(G)} we have an isomorphism $\fA\cong \Ahat$ given by the Fourier transform and we have already seen that $\Ahat$ is Morita equivalent to $\fB$. The algebra $\fB$ is isomorphic to $\fD$ via the above change of coordinates.
\end{proof}

We now compute the $K$-theory.  Define a new $C^*$-algebra as follows:
\[
\fC: = \{F \in C(\cY,  M_2(\C)) : F(y)\;  \mathrm{\;  is \; diagonal \; on} \; B, F(2) = 0\}.
\]
The map
\[
\fC \to M_2(\C), \quad \quad F \mapsto F(0)
\]
then fits into an  exact sequence of $C^*$-algebras
\[
0 \to \fD \to \fC \to M_2(\C) \to 0.
\]

This yields the six-term exact sequence 
\begin{equation}\label{6-term}
\begin{CD}
K_0(\fD) @>>> K_0(\fC) @>>>  \Z\\
@AAA &                       &             @VVV\\
0 @<<<  K_1(\fC)  @<<<  K_1(\fD)
\end{CD}
\end{equation}

Note that $\cY$ is a contractible space.   The following homotopy is well-adapted to the $C^*$-algebra $\fC$.  
Given $z = x + iy \in \cY$ define $h_t$ as follows:
\[
h_t(x + iy) = 
\left\{
\begin{array}{ll}
x + (1 - 2t)iy & \quad \quad   0 \leq t \leq 1/2\\
x  + (2 - x)(2t - 1) & \quad \quad 1/2 \leq t \leq 1
\end{array}
\right.
\]
This gives a homotopy equivalence from $\cY$ to the point $\{2\}$. Given a function $F\in\fC$, the composition $F(h_t(x+iy))$ also lies in $\fC$ since if $x+iy=x\in B$ then $h_t(x+iy)$ also lies in $B$, and hence $F(h_t(x+iy))$ is diagonal as required.


Thus $F\mapsto F\circ h_t$ induces a homotopy equivalence from $\fC$ to the zero $C^*$-algebra, 
 and hence the connecting maps in (\ref{6-term}) are isomorphisms.

Since $K$-theory is an invariant of strong Morita equivalence, we have the following result.
\begin{thm}  Let $\fA$ denote the reduced $C^*$-algebra of the universal cover of $\SL_2(\R)$. Then
\[
K_0(\fA) = 0, \quad \quad K_1(\fA) = \Z.
\]
\end{thm}

\mbox{}

Now recall that in Section \ref{the K-cycle} we considered the composition
\begin{equation}\label{restriction Morita}
\Ahat \hookrightarrow C_0(\cZ,\fK) \underset{\scriptstyle\text{Morita}}{\sim} C_0(\cZ)\twoheadrightarrow C_0(\R).
\end{equation}
At the level of $K$-theory this agrees with the composition
\[
\fD \hookrightarrow C_0(\cU,M_2(\C)) \underset{\scriptstyle\text{Morita}}{\sim} C_0(\cU)\twoheadrightarrow C_0((0,2)).
\]
which includes into the composition
\[
\fC \hookrightarrow C_0(\cY\setminus\{2\},M_2(\C)) \underset{\scriptstyle\text{Morita}}{\sim} C_0(\cY\setminus\{2\})\twoheadrightarrow C_0([0,2)).
\]
It follows that the 6-term exact sequence (\ref{6-term}) maps commutatively to the 6-term exact sequence
\begin{equation*}
\begin{CD}
K_0(C_0((0,2))) @>>> K_0(C_0([0,2))) @>>>  \Z\\
@AAA &                       &             @VVV\\
0 @<<<  K_1(C_0([0,2)))  @<<<  K_1(C_0((0,2)))
\end{CD}
\end{equation*}
from which we conclude that the map from $K_*(\fD)$ to $K_*(C_0((0,2)))$ is an isomorphism. Correspondingly $K_*(\Ahat)\to K_*(C_0(\R))$ is an isomorphism.

In Section \ref{the K-cycle} we showed that (\ref{restriction Morita}) takes the Kasparov triple $[\Ahat\oplus\Ahat,1,\Hdirac]$ to a generator of $KK^1(\C,C_0(\R))$. We thus conclude the following result.

\begin{thm}  Let $\fA$ denote the reduced $C^*$-algebra of the universal cover of $\SL_2(\R)$ and let $\Ahat\cong \fA$ denote its Fourier transform. Then $[\Ahat\oplus\Ahat,1,\hatD]$ is a generator of $KK^1(\C,\Ahat)$ and hence $[\fA\oplus\fA,1,D]$ is a generator of $KK^1(\C,\fA)$.
\end{thm}

\section{Dirac cohomology}
For the moment, let $G$ be a connected real reductive group with a maximal compact subgroup $K$.   Let $\mathfrak{g}$ be the Lie algebra of
$G$, and let $\widetilde{K}$ be the spin double cover of $K$ as defined in \cite[3.2.1]{HP}.   

\begin{defn} The Dirac cohomology of a $(\mathfrak{g}, K)$-module $X$ is the $\widetilde{K}$-module
\[
H_D(X) = Ker(D) / Im(D) \cap Ker(D)
\]
where $D$ is considered as an operator on $X \otimes S$ and $S$ is a space of spinors.
\end{defn}
  For unitary modules, the Dirac cohomology is just 
the kernel of $D$ on $X \otimes S$, see \cite[p.64]{HP}.  If $X$ is irreducible, then the Dirac cohomology is finite-dimensional \cite[p.62]{HP}. 

In the present context, where $G$ is the universal cover of $\SL_2(\R)$, it seems natural to define the Dirac cohomology, localised at a point in the reduced dual,
to be the kernel of the Dirac operator, localised at that point.

To determine the kernel, we have to ascertain when zero belongs to the local spectrum.   At this point, we invoke Theorem 3.2.   By inspection, the local spectra in
the discrete series do not contain $0$.   The spectrum of $D$ localised at the point  $(q,t)$ in the principal series is
\[
\{- \frac{1}{2} \pm \sqrt{1/4 + q + 2(\tau + k)(\tau + k - 1)} : k \in \Z\}
\]
This set contains $0$ if and only if $k = 0$.   The spectra containing $0$ are parametrised, in the $(q,\tau)$-plane, by the
parabolic arc
\begin{align*}\label{arc}
\mathcal{Q} : q + 2\tau(\tau - 1) = 0
\end{align*}
subject to the condition 
\[
\tau \in [1/2 - 1/\sqrt 8, 1/2 + 1/\sqrt 8].
\]  

At this point, we have to examine the action of the Dirac operator on the subspace $E_{\tau}$ defined in \S 3.   We obtain the following result.
\begin{thm} The Dirac cohomology is  concentrated in the parabolic arc $\mathcal{Q}$, and is given by a field of $2$-spinors 
over $\mathcal{Q}$.  The Dirac cohomology is one-dimensional at each point of $\mathcal{Q}$.  
With respect to the basis defined in \S 3, the Dirac cohomology is spanned, at the point $(q,\tau) \in \mathcal{Q}$, by the $2$-spinor
\[
\left(
\begin{array}{c}
\sqrt{1 - \tau}\\
- \sqrt {\tau}
\end{array}
\right)
\]
\end{thm}

The centre of the parabolic arc $\mathcal{Q}$ is the point $(1/2,1/2)$ in the $(q,\tau)$-plane,
i.e. the principal series representation $\pi_{1/2,1/2}$.   This representation factors through an irreducible representation $X$ of $\SL_2(\R)$ in the odd
principal series.    We note that the irreducible unitary module $X \otimes S$ satisfies Vogan's conjecture, see \cite[3.2.6]{HP}.

\end{document}